\PassOptionsToPackage{dvipsnames,svgnames,x11names}{xcolor}

\documentclass[12pt,reqno]{amsart}

\usepackage{amsmath}
\usepackage{amssymb}
\usepackage{amsthm}
\usepackage{bbm}
\usepackage{mathtools}
\usepackage{nicefrac}

\usepackage[main=english,brazilian]{babel}
\usepackage[utf8]{inputenc} 

\usepackage{a4wide}
\usepackage{microtype}

\usepackage[all]{xy}
\usepackage{tikz-cd}

\usepackage{aligned-overset}
\usepackage{enumitem}
\usepackage{multicol}
\usepackage{scalefnt}
\usepackage{xspace}

\usepackage{fixme}
\usepackage{url}

\usepackage{xcolor}
\usepackage{hyperref}
\usepackage{cleveref} 

\newtheorem{theorem}{Theorem}[section]
\newtheorem{proposition}[theorem]{Proposition}
\newtheorem{lemma}[theorem]{Lemma}
\newtheorem{corollary}[theorem]{Corollary}
\theoremstyle{definition}
\newtheorem{definition}[theorem]{Definition}
\newtheorem{remark}[theorem]{Remark}
\newtheorem{example}[theorem]{Example}

\newcommand*{\C}{\mathbb C}
\newcommand*{\Z}{\mathbb Z}
\newcommand*{\Star}{$^*$\nobreakdash}
\newcommand*{\id}{\operatorname{id}}
\DeclareMathOperator{\corr}{Corr}
\DeclareMathOperator{\rep}{Rep}
\DeclareMathOperator{\ev}{\text{ev}}		

\newcommand*{\aA}{\mathcal A}			
\newcommand*{\aB}{\mathcal B} 			
\newcommand*{\aC}{\mathcal C} 	
\newcommand*{\aD}{\mathcal D} 	

\title{Reduction of structure groups for \\ noncommutative principal bundles}
\author{Stefan Wagner}
\address{Blekinge Institute of Technology, Department of Mathematics,
Karlskrona, Sweden}
\email{stefan.wagner@bth.se}

\begin{document}

\begin{abstract}
Let $G$ be a compact group and let $H \subseteq G$ be a closed subgroup.
We study reduction of the structure group for noncommutative principal bundles, modelled by free C\Star-dynamical systems in the sense of Ellwood.
We first describe how induction from $H$ to~$G$ behaves on associated correspondences and their tensor structure.
Using reconstruction, this leads to a Tannakian characterization of reduction of the structure group in terms of the associated tensor functor.
For normal subgroups, we characterize reduction in terms of central equivariant homomorphisms and derive module- and $K$-theoretic obstructions.
We conclude with a noncommutative $O(n)$-example which admits no reduction to $SO(n)$.
\end{abstract}

\maketitle

\section{Introduction}
\label{sec:intro}

Let $G$ be a compact group, let $X$ be a compact Hausdorff space, and let $P \to X$ be a locally trivial principal $G$-bundle. 
If $H \subseteq G$ is a closed subgroup, a \emph{reduction of the structure group} from $G$ to $H$ is a principal $H$-subbundle $Q \subseteq P$. 
Such a reduction is equivalently characterized by a principal $H$-bundle $Q \to X$ together with an isomorphism of principal $G$-bundles
\[
    P \cong Q \times_H G.
\]
Indeed, a reduction $Q \subseteq P$ induces the canonical isomorphism $Q \times_H G \to P$, $[q,g] \mapsto qg$.
Thus reduction of the structure group may be viewed as an inverse problem for extension of the structure group.

Reductions of structure groups are ubiquitous in geometry and mathematical physics. 
They encode the presence of additional geometric structure on a bundle and are often subject to topological obstructions. 
A basic example, for $n \ge 1$, is provided by the inclusion 
\[
    SO(n) \subseteq O(n).
\]
A reduction of the orthonormal frame bundle of a Riemannian $n$-manifold from $O(n)$ to $SO(n)$ is equivalent to a choice of orientation. 
More generally, many geometric structures can be described in terms of reductions of frame bundles. 
In gauge theory, reductions also arise naturally in symmetry breaking, where the original symmetry group $G$ is reduced to a subgroup $H$.

A classical criterion for the existence of a reduction is obtained from the homogeneous space $G/H$. 
The quotient bundle $P/H \to X$ is a fibre bundle with typical fibre $G/H$.
The principal $G$-bundle $P \to X$ admits a reduction to $H$ if and only if this quotient bundle admits a global section $s: X \to P/H$.
Indeed, such a section selects continuously an $H$-orbit in each fibre of $P$, and the union of these orbits forms a principal $H$-subbundle of~$P$. 
Conversely, every reduction $Q\subseteq P$ determines a section by $s(x) := qH$, $q\in Q_x$.
This is well defined because $Q_x$ is an $H$-orbit.

The quotient bundle is canonically isomorphic to the associated bundle $P \times_G (G/H)$.
Thus the existence of a reduction can equivalently be formulated as the existence of a global section of the associated bundle with homogeneous fibre $G/H$.

The situation simplifies further when $H$ is normal in $G$. 
In this case $G/H$ is itself a compact group and $P/H \to X$ is a principal $G/H$-bundle. 
Consequently, a reduction from $G$ to $H$ exists if and only if this quotient bundle is trivial. 
For
example,
\[
    SO(n) \triangleleft O(n),
    \qquad
    O(n)/SO(n) \cong \mathbb{Z}_2.
\]
If $P \to X$ is the orthonormal frame bundle of a Riemannian $n$-manifold $X$, then $X$ is orientable if and only if the associated principal $\mathbb{Z}_2$-bundle $P/SO(n) \to X$ is trivial.

A further description of reduction, particularly well suited to the noncommutative setting, is obtained from associated vector bundles. 
Let $\sigma: G \to \mathcal{U}(V_\sigma)$ be a finite-dimensional unitary representation. 
The principal $G$-bundle $P$ then determines the associated vector bundle $P \times_G V_\sigma \to X$.
If $Q \subseteq P$ is a reduction to $H$, or equivalently if $P \cong Q \times_H G$, then there is a canonical isomorphism of vector bundles over $X$,
\[
    P \times_G V_\sigma
    \cong
    Q \times_H V_{\sigma\vert_H},
\]
where $\sigma\vert_H$ denotes the restriction of $\sigma$ to $H$. 
These isomorphisms are natural in $\sigma$ and compatible with tensor products. 

Denoting by $\rep(G)$ the C\Star-tensor category of finite-dimensional unitary representations of $G$, with intertwiners as morphisms, the
associated-vector-bundle construction for $P$ therefore factors through the restriction functor
\[
    \operatorname{Res}^G_H: \rep(G) \to \rep(H),
    \qquad
    \sigma \mapsto \sigma\vert_H.
\]

Noncommutative principal bundles lie at the intersection of several central themes in noncommutative geometry. 
From a geometric perspective, they extend the classical theory of principal and associated bundles to noncommutative spaces. 
In operator algebras and functional analysis, they arise naturally from free actions on C\Star-algebras and connect fixed point algebras, Hilbert C\Star-modules, correspondences, and $K$-theoretic constructions.
They also provide a natural setting for quantum symmetries, connections, and noncommutative gauge theories in mathematical physics. 
These different aspects have led to algebraic, C\Star-algebraic, and gauge-theoretic approaches to the study of noncommutative principal
bundles; see, for example,~\cite{BaCoHa15,Cac23,CacMes19,Du95,FaLa24,LaSu05,SchWa17,SchWa20,Wa25,Wa23} and the references therein.

The purpose of this article is to study reduction of the structure group in this setting. 
We model noncommutative principal bundles by free C\Star-dynamical systems in the sense of Ellwood.

Let $\aA$ be a unital C\Star-algebra and let $G$ be a compact group.
In this paper, a strongly continuous action $\alpha: G \to \operatorname{Aut}(\aA)$, $g \mapsto \alpha_g$ will be referred to as a \emph{C\Star-dynamical system} and denoted by $(\aA,G,\alpha)$. 

A C\Star-dynamical system $(\aA,G,\alpha)$ is called \emph{free} if the Ellwood map
\[
    \Phi:
    \aA \otimes_{\mathrm{alg}} \aA \to C(G,\aA),
    \qquad
    \Phi(x \otimes y)(g) := x \alpha_g(y),
\]
has dense range with respect to the canonical C\Star-norm on $C(G,\aA)$. 
This condition was introduced by Ellwood in the setting of compact quantum group actions~\cite{Ell00} and, for classical compact group actions, is equivalent to Rieffel's saturatedness condition~\cite{Rieffel91} and to the Peter--Weyl--Galois condition~\cite{BaCoHa15}.

This notion agrees with classical freeness.
Indeed, if a compact group $G$ acts on a compact space $P$, then the action is free if and only if the induced C\Star-dynamical system $(C(P),G,\alpha)$ is free in the above sense; see \cite[Prop.~7.1.12 and Thm.~7.2.6]{Phi87}. 
Free C\Star-dynamical systems therefore provide a natural framework for noncommutative principal bundles.

A free C\Star-dynamical system $(\aA,G,\alpha)$ with fixed point algebra $\aB$ determines a unitary tensor functor
\[
    \Gamma_\aA: \rep(G) \to \corr(\aB),
\]
where $\corr(\aB)$ denotes the C\Star-tensor category of correspondences over $\aB$. 
The functor is constructed from the associated correspondences of the action. 
Its construction, together with the corresponding reconstruction result, is recalled in Section~\ref{sec:induction}.

Now let $H \subseteq G$ be a closed subgroup and let $(\aC,H,\gamma)$ be a free C\Star-dynamical system with fixed point algebra $\aB$. 
The noncommutative analogue of extension of the structure group is given by the induced C\Star-algebra
\begin{equation}
\label{eq:ind}
    \operatorname{Ind}_H^G(\aC) := C(G,\aC)^H,
\end{equation}
where
\[
    C(G,\aC)^H
    :=
    \{
        f \in C(G,\aC):
        (\forall g \in G)\;
        (\forall h \in H)\;
        f(h^{-1}g) = \gamma_{h^{-1}}(f(g))
    \}.
\]
The group $G$ acts on $\operatorname{Ind}_H^G(\aC)$ by right translation,
\[
    (\beta_kf)(g) := f(gk),
    \qquad g, k\in G.
\]

A key observation is that induction has a particularly simple description on the level of associated correspondences. 
For each $\sigma \in \rep(G)$, evaluation at the identity induces a canonical
unitary isomorphism
\[
    \Gamma_{\operatorname{Ind}_H^G(\aC)}(\sigma)
    \cong
    \Gamma_\aC(\sigma\vert_H).
\]
These isomorphisms are natural in $\sigma$ and compatible with tensor products, and therefore give rise to a unitary monoidal natural
isomorphism
\[
    \Gamma_{\operatorname{Ind}_H^G(\aC)}
    \cong
    \Gamma_\aC \circ \operatorname{Res}^G_H.
\]
Thus the classical relation between associated vector bundles has a direct noncommutative counterpart at the level of associated correspondences.

One of the main results of this article shows that this relation characterizes reduction: a free C\Star-dynamical system $(\aA,G,\alpha)$ admits a reduction to $(\aC,H,\gamma)$ precisely when its associated tensor functor factors, up to unitary monoidal natural isomorphism, through the restriction functor
\[
    \operatorname{Res}^G_H: \rep(G) \to \rep(H).
\]
Thus reduction of the structure group can be characterized entirely in terms of the associated tensor-functorial data.

For a closed normal subgroup $H \trianglelefteq G$, a further characterization is obtained in terms of the center of the total algebra.
More precisely, a reduction to $H$ exists if and only if there is a unital $G$-equivariant \Star-homomorphism
\[
    j: C(G/H) \to Z(\aA),
\]
where $G$ acts on $C(G/H)$ through the quotient by right translation.
Equivalently, the residual $G/H$-system admits a trivialization whose canonical copy of $C(G/H)$ is central in $\aA$.
In particular, reduction implies triviality of the residual system, which yields module- and $K$-theoretic obstructions.

\subsection*{Organization of the paper}

Section~2 studies induction from a closed subgroup $H \subseteq G$.
After recalling the associated tensor functor of a free C\Star-dynamical system, freeness of the induced C\Star-dynamical system is recalled,
and Theorem~\ref{thm:induction-tensor-functor} shows that evaluation
at the identity identifies its associated tensor functor with that of
the original $H$-C\Star-dynamical system composed with restriction of representations.

Section~3 introduces reduction of the structure group.
A rigidity argument identifies~reductions with induced C\Star-dynamical systems, and Theorem~\ref{thm:reduction-tannakian} gives the main tensor-functorial characterization of reduction.
For normal subgroups, Theorem~\ref{thm:normal-reduction} characterizes reduction through central equivariant homomorphisms and shows that the residual $G/H$-system is trivial whenever a reduction exists.
This leads to module- and $K$-theoretic obstructions and to a noncommutative $O(n)$-example admitting no reduction to $SO(n)$.

The paper concludes with open problems concerning noncommutative principal $O(n)$-bundles and bundles with homogeneous fibres.

\section{Induction and associated correspondences}
\label{sec:induction}

We begin by recalling the tensor functor associated with a free C\Star-dynamical system. 
Let $(\aA,G,\alpha)$ be a free C\Star-dynamical system with fixed point algebra
$\aB := \aA^G$. 
Denote by $\corr(\aB)$ the C\Star-tensor category of correspondences over $\aB$, with adjointable $\aB$-bimodule maps as morphisms and the interior tensor product over $\aB$ as tensor product.

Let $\sigma\in\rep(G)$. 
Define
\[
    \Gamma_\aA(\sigma)
    :=
    \{
        x \in \aA \otimes V_\sigma:
        (\forall g \in G)\;
        (\alpha_g \otimes \sigma_g)(x) = x
    \}.
\]
The space $\Gamma_\aA(\sigma)$ is naturally a correspondence over $\aB$. 
Its left and right $\aB$-actions are induced by multiplication in $\aA$, while the $\aB$-valued inner product is obtained by
restricting the canonical right $\aA$-valued inner product on
$\aA \otimes V_\sigma$, determined on elementary tensors by
\[
    \langle a \otimes v,b \otimes w\rangle_\aA
    :=
    \langle v,w\rangle_{V_\sigma} a^*b
\]
for all $a,b \in \aA$ and $v,w \in V_\sigma$. 
Indeed, for $x,y \in \Gamma_\aA(\sigma)$, the element $\langle x,y\rangle_\aA$ is $G$-invariant and therefore belongs to $\aB$.

The assignment $\sigma \mapsto \Gamma_\aA(\sigma)$ extends to a unitary tensor functor
\[
    \Gamma_\aA: \rep(G) \to \corr(\aB).
\]
On morphisms, the functor is induced by the corresponding intertwiners of representations, while its tensor structure is determined by multiplication in $\aA$. 
Conversely, free C\Star-dynamical systems with structure group $G$ and fixed point algebra $\aB$ can be reconstructed, up to $G$-equivariant \Star-isomorphism over $\aB$, from the corresponding tensor-functorial data; see \cite[Sec.~5]{SchWa17} and also~\cite{Du95}. 
This reconstruction principle will be used below in the characterization of reduction of the structure group.

We now turn to induction.
Throughout this section, let $H \subseteq G$ be a closed subgroup and let $(\aC,H,\gamma)$ be a free C\Star-dynamical system with fixed point algebra $\aB := \aC^H$.
The induced C\Star-algebra
\[
    \aD := \operatorname{Ind}_H^G(\aC),
\]
defined in Equation~\ref{eq:ind}, is equipped with the right translation
action $\beta$ of $G$.
The resulting C\Star-dynamical system will be denoted by $(\aD,G,\beta)$.

The following standard result shows that freeness is preserved under induction:

\begin{proposition}
\label{prop:induction-free}
The C\Star-dynamical system $(\aD,G,\beta)$ is free and its fixed point algebra is canonically identified with $\aB$.
\end{proposition}

\begin{proof}
Freeness follows from the induction result~\cite[Thm.~3.21]{SchWa17}. 
It remains to identify the fixed point algebra.

Suppose that $f \in \aD$ is fixed by $\beta$. 
Then $f(gk)=f(g)$ for all $g,k \in G$, and hence $f$ is constant, say $f(g)=c_0$. 
The defining covariance relation for $\aD$ gives
\[
    c_0 = f(h^{-1}) = \gamma_{h^{-1}}(f(e)) = \gamma_{h^{-1}}(c_0)
\]
for all $h \in H$. 
It follows that $c_0 \in \aC^H = \aB$. 
Conversely, each $b \in \aB$ determines a constant function in $\aD$ which is fixed by $\beta$. 
Therefore $\aD^G = \aB$.
\end{proof}

The associated correspondences of $(\aD,G,\beta)$ can now be described explicitly:

\begin{theorem}
\label{thm:induction-tensor-functor}
For each $\sigma \in \rep(G)$, evaluation at the identity element $e \in G$ defines a unitary isomorphism of $\aB$-correspondences
\[
    \ev_e^\sigma: \Gamma_\aD(\sigma) \to \Gamma_\aC(\sigma\vert_H),
    \qquad
    F \mapsto F(e).
\]
These isomorphisms assemble into a unitary monoidal natural isomorphism
\[
    \varepsilon: \Gamma_\aD \Longrightarrow \Gamma_\aC \circ \operatorname{Res}^G_H,
\]
whose component at $\sigma$ is $\varepsilon_\sigma=\ev_e^\sigma$.
\end{theorem}
\begin{proof}
Let $\sigma \in \rep(G)$.
Given $F \in \Gamma_\aD(\sigma)$, the inclusion $\aD \otimes V_\sigma \subseteq C(G,\aC \otimes V_\sigma)$ allows $F$ to be regarded as a continuous function $F: G \to \aC \otimes V_\sigma$ whose $G$-invariance is expressed by $(\id_\aC \otimes \sigma_k)F(gk) = F(g)$ for all $g,k \in G$.
Hence $F(g) = (\id_\aC \otimes \sigma_g^{-1})F(e)$ for all $g \in G$, and so $F$ is completely determined by its value at the identity.

Set $x := F(e)$. 
The covariance relation defining $\aD$ implies that $F(h) = (\gamma_h \otimes \id_{V_\sigma})x$ for all $h \in H$.
On the other hand, $F(h) = (\id_\aC \otimes \sigma_{h^{-1}})x$ for all $h \in H$.
Therefore $(\gamma_h \otimes \sigma_h)x = x$ for all $h \in H$, showing that $x \in \Gamma_\aC(\sigma\vert_H)$.
Consequently, $\ev_e^\sigma$ is well defined.

Conversely, let $x \in \Gamma_\aC(\sigma\vert_H)$.
Define
\[
    F_x(g) := (\id_\aC \otimes \sigma_g^{-1})x,
    \qquad g \in G.
\]
The corresponding map $F_x: G \to \aC \otimes V_\sigma$ is continuous by continuity of $\sigma$. 
The $H$-invariance of $x$ gives $F_x(h^{-1}g) = (\gamma_{h^{-1}} \otimes \id_{V_\sigma})F_x(g)$ for all $h \in H$ and $g \in G$.
Hence $F_x \in \aD \otimes V_\sigma$. 
Furthermore, $(\id_\aC \otimes \sigma_k) F_x(gk) = F_x(g)$ for all $g,k \in G$, which establishes that $F_x \in \Gamma_\aD(\sigma)$. 
Since $F_x(e) = x$, the map $x \mapsto F_x$ is the inverse of $\ev_e^\sigma$.
The map $\ev_e^\sigma$ is $\aB$-bilinear, since elements of $\aB$ are identified with constant functions.
Moreover, for $F,K \in \Gamma_\aD(\sigma)$,
\[
    \langle F,K\rangle_\aB = \langle F(e),K(e)\rangle_\aB,
\]
where the inner products are those of the respective $\aB$-correspondences.
Thus $\ev_e^\sigma$ is unitary.

Let $\tau \in \rep(G)$. 
For each intertwiner $T:V_\sigma\to V_\tau$, naturality follows from
\[
    (\id_\aC \otimes T)F(e) = ((\id_\aC \otimes T)F)(e).
\]
Compatibility with the tensor structure follows from
\[
    \ev_e^{\sigma\otimes\tau}(F_{12}K_{13}) = F(e)_{12} K(e)_{13}
\]
for $F \in \Gamma_\aD(\sigma)$ and $K \in \Gamma_\aD(\tau)$. 
Thus $(\ev_e^\sigma)_{\sigma \in \rep(G)}$ defines a unitary monoidal natural~isomorphism.
\end{proof}

\begin{remark}
\label{rem:classical-associated}
Suppose that $\aC=C(Q)$, where $Q\to X$ is a principal $H$-bundle.
By Gelfand duality, the induced algebra $\operatorname{Ind}_H^G(\aC)$ is canonically isomorphic to $C(Q \times_H G)$. 
Under the Serre--Swan correspondence, Theorem~\ref{thm:induction-tensor-functor} therefore recovers the classical isomorphism of associated vector bundles
\[
    (Q \times_H G) \times_G V_\sigma \cong Q \times_H V_{\sigma\vert_H}.
\]
Thus Theorem~\ref{thm:induction-tensor-functor} is the noncommutative counterpart of the compatibility between extension of the structure group and the formation of associated vector bundles.
\end{remark}

\section{Reduction of the structure group}
\label{sec:reduction}

Throughout this section, let $H\subseteq G$ be a closed subgroup and let $(\aA,G,\alpha)$ be a free C\Star-dynamical system with fixed point algebra $\aB:=\aA^G$.

\begin{definition}
\label{def:reduction}
An \emph{$H$-reduction} of $(\aA,G,\alpha)$ is a pair $((\aC,H,\gamma),\phi)$, where $(\aC,H,\gamma)$ is a free C\Star-dynamical system with fixed point algebra $\aB$ and $\phi: \aA \to \aC$ is a unital \Star-homomorphism satisfying $\phi\vert_\aB = \id_\aB$ and $\phi \circ \alpha_h = \gamma_h \circ \phi$ for all $h \in H$.
\end{definition}

\begin{remark}
\label{rem:classical-reduction-map}
Suppose that $\aA=C(P)$ and $\aC=C(Q)$, where $P\to X$ and $Q\to X$ are principal $G$- and $H$-bundles, respectively. 
By Gelfand duality, a unital \Star-homomorphism $\phi: C(P) \to C(Q)$ is induced by a continuous map $Q \to P$. 
The conditions in Definition~\ref{def:reduction} imply that this map is $H$-equivariant and covers the identity on $X$. 
It therefore identifies $Q$ with a principal $H$-subbundle of $P$. 
Thus Definition~\ref{def:reduction} is the contravariant C\Star-algebraic counterpart of the classical notion of reduction of the structure group.
\end{remark}

The following rigidity property of free C\Star-dynamical systems will be used below. 
It is recalled from \cite[Lem.~3.10]{SchWa24}.

\begin{lemma}
\label{lem:free-subalgebra-rigidity}
Let $(\aD,G,\beta)$ be a free C\Star-dynamical system with fixed point algebra $\aB$, and let $\aD_0 \subseteq \aD$ be a unital $G$-invariant C\Star-subalgebra. 
If the restricted C\Star-dynamical system $(\aD_0,G,\beta\vert_{\aD_0})$ is free and $\aD_0^G = \aD^G = \aB$, then $\aD_0 = \aD$.
\end{lemma}

The following convenient form is an immediate consequence:

\begin{corollary}
\label{cor:equivariant-rigidity}
Let $(\aA,G,\alpha)$ and $(\aD,G,\beta)$ be free C\Star-dynamical systems with $\aA^G = \aD^G = \aB$.
If $\phi: \aA \to \aD$ is an injective unital $G$-equivariant \Star-homomorphism satisfying $\phi\vert_\aB = \id_\aB$, then $\phi$ is a \Star-isomorphism.
\end{corollary}

\begin{proof}
The image $\phi(\aA)$ is a unital $G$-invariant C\Star-subalgebra of $\aD$. 
Since $\phi$ is injective, the restricted C\Star-dynamical system $(\phi(\aA),G,\beta\vert_{\phi(\aA)})$ is free, and $\phi(\aA)^G = \phi(\aA^G) = \aB = \aD^G$.
Lemma~\ref{lem:free-subalgebra-rigidity} therefore implies that $\phi(\aA)=\aD$.
\end{proof}

The following result identifies a reduction with an induced C\Star-dynamical system:

\begin{lemma}
\label{lem:reduction-induced}
Let $(\aC,H,\gamma)$ together with $\phi: \aA \to \aC$ be an $H$-reduction of $(\aA,G,\alpha)$. 
Then
\[
    \iota: \aA \to \operatorname{Ind}_H^G(\aC),
    \qquad
    \iota(a)(g) := \phi(\alpha_g(a)),
\]
defines a $G$-equivariant \Star-isomorphism satisfying $\iota\vert_\aB = \id_\aB$.
\end{lemma}
\begin{proof}
Let $a \in \aA$.
Continuity of the action $\alpha$ implies that $\iota(a) \in C(G,\aC)$ and, for all $h\in H$ and $g \in G$,
\[
    \iota(a)(h^{-1}g)
    =
    \phi(\alpha_{h^{-1}}(\alpha_g(a)))
    =
    \gamma_{h^{-1}}(\phi(\alpha_g(a)))
    =
    \gamma_{h^{-1}}(\iota(a)(g)).
\]
Consequently, $\iota(a) \in \operatorname{Ind}_H^G(\aC)$. 

Pointwise operations show that $\iota$ is a unital \Star-homomorphism.
For all $g,k \in G$,
\[
    (\beta_k\iota(a))(g)
    =
    \iota(a)(gk)
    =
    \phi(\alpha_g(\alpha_k(a)))
    =
    \iota(\alpha_k(a))(g).
\]
Hence $\iota$ is $G$-equivariant. 
For all $b \in \aB$ and $g \in G$, $\iota(b)(g) = \phi(\alpha_g(b)) = b$, so $\iota\vert_\aB = \id_\aB$.

To prove injectivity, suppose that $\iota(a)=0$. 
Then $\phi(\alpha_g(a^*a))=0$ for all $g \in G$. 
Let
\[
    E_\aA: \aA \to \aB,
    \qquad
    E_\aA(x) := \int_G\alpha_g(x)\,dg,
\]
denote the canonical conditional expectation. 
Since $\phi\vert_\aB = \id_\aB$,
\[
    E_\aA(a^*a)
    =
    \phi(E_\aA(a^*a))
    =
    \int_G\phi(\alpha_g(a^*a))\,dg
    =
    0.
\]
Faithfulness of $E_\aA$ gives $a=0$. 
Thus $\iota$ is injective.

By Proposition~\ref{prop:induction-free}, $(\operatorname{Ind}_H^G(\aC),G,\beta)$ is free and has fixed point algebra $\aB$. 
The hypotheses of Corollary~\ref{cor:equivariant-rigidity} are therefore satisfied, and $\iota$ is surjective.
This completes the proof.
\end{proof}

The preceding result leads to a characterization of reduction in terms of the associated tensor functors:

\begin{theorem}
\label{thm:reduction-tannakian}
Let $(\aC,H,\gamma)$ be a free C\Star-dynamical system with fixed point algebra $\aB$.
The following statements are equivalent:
\begin{enumerate}[label=\textup{(\alph*)}]
    \item $(\aA,G,\alpha)$ admits an $H$-reduction to $(\aC,H,\gamma)$.
    \item There exists a unitary monoidal natural isomorphism
    \[
        \eta: \Gamma_\aA \Longrightarrow \Gamma_\aC \circ \operatorname{Res}^G_H
    \]
    whose component at the trivial representation is $\eta_{1_G} = \id_\aB$.
\end{enumerate}
\end{theorem}

\begin{proof}
Suppose first that \textup{(a)} holds. 
By Lemma~\ref{lem:reduction-induced}, $\aA$ and $\operatorname{Ind}_H^G(\aC)$ are $G$-equivariantly \Star-isomorphic over $\aB$, and the resulting isomorphism induces a
unitary monoidal natural isomorphism $\Gamma_\aA \cong \Gamma_{\operatorname{Ind}_H^G(\aC)}$ whose component at the trivial representation is $\id_\aB$.
Composing this with the unitary monoidal natural isomorphism from Theorem~\ref{thm:induction-tensor-functor} proves \textup{(b)}.

Conversely, suppose that \textup{(b)} holds. 
By Theorem~\ref{thm:induction-tensor-functor}, there is a unitary monoidal natural isomorphism
\[
    \varepsilon: \Gamma_{\operatorname{Ind}_H^G(\aC)} \Longrightarrow \Gamma_\aC \circ \operatorname{Res}^G_H.
\]
Hence
\[
    \varepsilon^{-1} \circ \eta: \Gamma_\aA \Longrightarrow \Gamma_{\operatorname{Ind}_H^G(\aC)}
\]
is a unitary monoidal natural isomorphism whose component at the trivial representation is $\id_\aB$.

By the reconstruction theorem for compact (quantum) group actions~\cite[Thm.~2.2]{Ne13}, this unitary monoidal natural isomorphism induces a $G$-equivariant \Star-isomorphism $\Phi: \aA \to \operatorname{Ind}_H^G(\aC)$ satisfying $\Phi\vert_\aB = \id_\aB$.
Let
\[
    \ev_e: \operatorname{Ind}_H^G(\aC) \to \aC,
    \qquad
    f \mapsto f(e),
\]
denote evaluation at the identity and set $\phi := \ev_e \circ \Phi$.
Then $\phi$ is a unital \Star-homomorphism satisfying $\phi\vert_\aB=\id_\aB$.
The $G$-equivariance of $\Phi$ gives
\[
    \phi(\alpha_h(a))
    =
    \ev_e(\beta_h(\Phi(a)))
    =
    \Phi(a)(h)
\]
for all $h \in H$ and $a \in \aA$, while the defining covariance relation of $\operatorname{Ind}_H^G(\aC)$ gives $f(h) = \gamma_h(f(e))$ for all $h \in H$ and $f \in \operatorname{Ind}_H^G(\aC)$.
Consequently,
\[
    \phi(\alpha_h(a))
    =
    \gamma_h(\Phi(a)(e))
    =
    \gamma_h(\phi(a))
\]
for all $h \in H$ and $a \in \aA$.
Thus $\phi$ is an $H$-reduction.
\end{proof}

Theorem~\ref{thm:reduction-tannakian} shows that reduction of the structure group is completely encoded by the associated tensor functor.
More precisely, a reduction from $G$ to $H$ exists precisely when there exists a free C\Star-dynamical system $(\aC,H,\gamma)$ with fixed point algebra $\aB$ such that $\Gamma_\aA$ is unitarily monoidally naturally
isomorphic to
\[
    \Gamma_\aC \circ \operatorname{Res}^G_H,
\]
with the isomorphism restricting to $\id_\aB$ at the trivial representation.
This is the noncommutative counterpart of the classical compatibility between reduction of the structure group and associated vector bundles.

\subsection{Reductions to normal subgroups}

Suppose now that $H \trianglelefteq G$ is a closed normal subgroup. 
Set $Q := G/H$ and let $q:G \to Q$ denote the quotient homomorphism.
Let $\rho$ denote the right translation action of $Q$ on $C(Q)$.
The fixed point algebra $\aA^H$ carries a residual action of $Q$ defined by $\bar\alpha_{q(g)}(a) := \alpha_g(a)$ for all $g \in G$ and $a \in \aA^H$. 
The resulting C\Star-dynamical system $(\aA^H,Q,\bar\alpha)$ is free by \cite[Prop.~3.18]{SchWa17}, and its fixed point algebra is $(\aA^H)^Q = \aA^G = \aB$.

Classically, if $P \to X$ is a principal $G$-bundle, then $P/H \to X$ is a principal $Q$-bundle.
A reduction of $P$ to $H$ exists if and only if this quotient bundle is trivial. 
In the noncommutative setting, reduction can be characterized by a central equivariant embedding of $C(Q)$.

\begin{theorem}\label{thm:normal-reduction}
Let $H \trianglelefteq G$ be a closed normal subgroup and set $Q := G/H$. 
The following statements are equivalent:
\begin{enumerate}[label=\textup{(\alph*)}]
    \item $(\aA,G,\alpha)$ admits an $H$-reduction.
    \item There exists a unital $G$-equivariant \Star-homomorphism $j: C(Q) \to Z(\aA)$,
    where $G$ acts on $C(Q)$ through $q$ by right translation.
    \item There exists a unital $Q$-equivariant \Star-isomorphism $\psi: \aA^H \to \aB \otimes C(Q)$
    satisfying $\psi(b) = b\otimes 1$ for all $b \in \aB$ and $\psi^{-1}(1 \otimes C(Q)) \subseteq Z(\aA)$.
\end{enumerate}
\end{theorem}

\begin{proof}
Suppose first that~(a) holds, and let $(\aC,H,\gamma)$ together with $\phi: \aA \to \aC$ be an $H$-reduction.
By Lemma~\ref{lem:reduction-induced}, there exists a $G$-equivariant \Star-isomorphism over $\aB$,
\[
    \iota: \aA \to \operatorname{Ind}_H^G(\aC),
    \qquad
    \iota(a)(g) := \phi(\alpha_g(a)).
\]
The $H$-fixed point algebra of the induced system is canonically identified with $C(Q,\aB)$.
Indeed, if $f \in \operatorname{Ind}_H^G(\aC)$ is $H$-invariant, then $f(gh) = f(g)$ for all $g \in G$ and $h \in H$.
Since $H$ is normal, this also implies $f(h^{-1}g) = f(g)$ for all $g \in G$ and $h \in H$.
The covariance relation therefore gives $f(g) \in \aC^H = \aB$ for all $g \in G$.
Conversely, every element of $C(Q,\aB)$ determines an $H$-fixed element of the induced algebra.
In consequence,
\[
    (\operatorname{Ind}_H^G(\aC))^H \cong C(Q,\aB) \cong \aB \otimes C(Q).
\]
Restricting $\iota$ to the $H$-fixed point algebras and using the preceding identification yields a $Q$-equivariant \Star-isomorphism $\psi: \aA^H \to \aB \otimes C(Q)$ satisfying $\psi(b) = b \otimes 1$ for all $b \in \aB$.

Let $f \in C(Q)$ and define
\[
    m_f(g) := f(q(g))1_{\aC},
    \qquad g \in G.
\]
Then $m_f \in \operatorname{Ind}_H^G(\aC)$.
Moreover, $m_f$ is central, as its values are scalar multiples of $1_{\aC}$, and $H$-invariant, as $m_f(gh)=m_f(g)$ for all $g \in G$ and $h \in H$.
It follows that
\[
    \psi^{-1}(1 \otimes f) = \iota^{-1}(m_f) \in Z(\aA).
\]
This proves~(c).

Suppose next that~(c) holds. 
Define
\[
    j: C(Q) \to Z(\aA),
    \qquad
    j(f) := \psi^{-1}(1 \otimes f).
\]
Then $j$ is a unital \Star-homomorphism, and the $Q$-equivariance of $\psi$ gives $\alpha_g(j(f)) = j(\rho_{q(g)}f)$ for all $g \in G$ and $f \in C(Q)$. 
Thus~(b) holds.

Finally, suppose that~(b) holds.
The $G$-equivariance of $j$ implies
\[
    \alpha_h(j(f)) = j(\rho_{q(h)}f) = j(f)
\]
for all $h \in H$ and $f \in C(Q)$.
Hence $j(C(Q)) \subseteq \aA^H \cap Z(\aA)$.

As an intermediate step in the construction of an $H$-reduction, it is shown that $(\aA^H,Q,\bar\alpha)$ is trivial.
Define
\[
    \varphi: \aB \otimes C(Q) \to \aA^H,
    \qquad
    \varphi(b \otimes f) := bj(f).
\]
Since $j(C(Q))$ is central and $C(Q)$ is commutative, this defines a unital \Star-homomorphism.
Moreover, $\varphi$ is $Q$-equivariant and $\varphi(b \otimes 1) = b$ for all $b \in \aB$.

Let
\[
    E_Q: \aA^H \to \aB,
    \qquad
    E_Q(a) := \int_Q  \bar\alpha_q(a)\,dq,
\]
denote the canonical conditional expectation associated with $(\aA^H,Q,\bar\alpha)$, and let
\[
    H_Q: C(Q) \to \C,
    \qquad
    H_Q(f) := \int_Q f(q)\,dq,
\]
denote integration with respect to normalized Haar measure on $Q$.
By the $G$-equivariance of $j$ and the invariance of Haar measure,
\[
    E_Q(j(f)) = \int_Q\bar\alpha_q(j(f))\,dq = j\left(\int_Q\rho_q f\,dq\right) = H_Q(f)1_{\aB}
\]
for all $f \in C(Q)$.
Consequently,
\[
    E_Q(\varphi(x)^* \varphi(x)) = (\operatorname{id}_{\aB} \otimes H_Q)(x^*x)
\]
for all $x \in \aB \otimes C(Q)$.
Since $\operatorname{id}_{\aB} \otimes H_Q$ is faithful, $\varphi$ is injective.
The image of $\varphi$ is therefore a unital $Q$-invariant C\Star-subalgebra of $\aA^H$ on which
the restricted action is free, with fixed point algebra $\aB$.
Since $(\aA^H,Q,\bar\alpha)$ is also free and has the same fixed point algebra, Lemma~\ref{lem:free-subalgebra-rigidity} implies that $\varphi$ is surjective. 

With this identification established, let
\[
    I_e := \{f \in C(Q): f(e_Q) = 0\},
    \qquad
    J_e := \overline{j(I_e)\aA}.
\]
Since $j(C(Q)) \subseteq Z(\aA)$, $J_e$ is a closed two-sided ideal of $\aA$. 
Moreover, $j(C(Q)) \subseteq \aA^H$, and so $J_e$ is $H$-invariant.
It follows that the action of $H$ descends to an action $\gamma$ on 
\[
    \aC := \aA/J_e.
\]
Let $\phi: \aA \to \aC$ denote the quotient map.
By construction, $\phi \circ \alpha_h = \gamma_h \circ \phi$ for all $h \in H$.

It remains to show that $(\aC,H,\gamma)$ is free and has fixed point algebra $\aB$.
Let
\[
    E_H: \aA \to \aA^H,
    \qquad
    E_H(a) := \int_H\alpha_h(a)\,dh,
\]
denote the canonical conditional expectation associated with the restricted $H$-action.
By the $\aA^H$-bimodularity and continuity of $E_H$, $E_H(J_e) \subseteq \overline{j(I_e)\aA^H}$.
Since $E_H$ restricts to the identity on $\aA^H$ and $\overline{j(I_e)\aA^H} \subseteq J_e \cap \aA^H$, it follows that
\[
    J_e \cap \aA^H = \overline{j(I_e)\aA^H}.
\]

Averaging over $H$ yields $\aC^H \cong \aA^H/(J_e \cap \aA^H)$.
Under the isomorphism $\varphi$, the ideal~$J_e \cap \aA^H$ corresponds to $\aB \otimes I_e$.
Consequently,
\[
    \aC^H \cong (\aB \otimes C(Q))/(\aB \otimes I_e) \cong \aB,
\]
where the last isomorphism is induced by evaluation at $e_Q$. 
Since $\varphi(b \otimes 1) = b$ for all $b \in \aB$, the restriction $\phi\vert_{\aB}$ corresponds to $\operatorname{id}_{\aB}$ under this identification.

To establish freeness, recall that the restriction of the free $G$-action on $\aA$ to $H$ is free (see \cite[Prop.~3.17]{SchWa17}).
The surjective $H$-equivariant \Star-homomorphism $\phi: \aA \to \aC$ induces a surjective map $C(H,\aA) \to C(H,\aC)$ that carries the range of the Ellwood map for $(\aA,H,\alpha\vert_H)$ onto the range of that for $(\aC,H,\gamma)$.
Consequently, the latter has dense range, and hence $(\aC,H,\gamma)$ is free.

Thus $(\aC,H,\gamma)$ together with $\phi$ is an $H$-reduction. 
This proves~(a) and completes the equivalence.
\end{proof}

The homomorphism $j$ in Theorem~\ref{thm:normal-reduction}\textup{(b)}
is automatically injective, because its kernel is a $Q$-invariant ideal of $C(Q)$ and the right translation action of $Q$ is transitive.

Thus, reduction to a normal subgroup is characterized by the existence of a central equivariant embedding of $C(Q)$ into $\aA$.
In particular, triviality of the residual $Q$-system
is a necessary condition for reduction:

\begin{corollary}\label{cor:residual-trivial}
Suppose that $(\aA,G,\alpha)$ admits an $H$-reduction. 
Then $(\aA^H,Q,\bar\alpha)$ is $Q$-equivariantly isomorphic over $\aB$ to $(\aB \otimes C(Q),Q,\operatorname{id}_{\aB} \otimes \rho)$.
\end{corollary}

\begin{remark}
By Theorem~\ref{thm:normal-reduction}, reduction to $H$ is characterized by the existence of a $Q$-equivariant \Star-isomorphism $\psi: \aA^H \to \aB \otimes C(Q)$ over $\aB$ satisfying $\psi^{-1}(1 \otimes C(Q)) \subseteq Z(\aA)$.
In the commutative case, the centrality condition is automatic.
In general, however, triviality of the residual $Q$-C\Star-dynamical system does not imply the existence of an $H$-reduction, as the following example shows.
\end{remark}

\begin{example}
Let $G := \Z_2 \times \Z_2$ and let $\aA := M_2(\C)$.
Consider the unitaries
\[
    u:=
    \begin{pmatrix}
        1 & 0\\
        0 & -1
    \end{pmatrix},
    \qquad
    v:=
    \begin{pmatrix}
        0 & 1\\
        1 & 0
    \end{pmatrix}.
\]
Since $uv=-vu$, the automorphisms $\operatorname{Ad}(u)$ and $\operatorname{Ad}(v)$ commute and define an action $\alpha$ of $G$
on $\aA$.
The corresponding spectral subspaces are one-dimensional and spanned by unitaries.
Consequently, the action is free and $\aB := \aA^G = \C 1_\aA$.

Let $H \cong \Z_2$ be the subgroup generated by the first factor and set $Q := G/H \cong \Z_2$.
Then
\[
    \aA^H
    =
    \left\{
        \begin{pmatrix}
            a & 0\\
            0 & b
        \end{pmatrix}
        :a,b \in \C
    \right\}
    \cong C(Q).
\]
The residual $Q$-action interchanges the diagonal entries.
Hence $(\aA^H,Q,\bar\alpha) \cong (C(Q),Q,\rho)$.

Nevertheless, $(\aA,G,\alpha)$ admits no $H$-reduction.
Indeed, by Theorem~\ref{thm:normal-reduction}, such a reduction would give a unital $Q$-equivariant \Star-homomorphism $j:C(Q) \to Z(\aA) = \C 1_\aA$.
No such homomorphism exists, because every unital \Star-homomorphism from $C(Q)$ to $\C$ is evaluation at a point, and no point of
$Q$ is fixed under all right translations.
Thus triviality of the residual system is not sufficient for reduction.
\end{example}

The preceding corollary has a particularly simple formulation in terms of associated correspondences.
Let
\[
    \Gamma_{\aA^H}: \rep(Q) \to \corr(\aB)
\]
denote the tensor functor associated with the residual $Q$-action.

Suppose that $(\aA,G,\alpha)$ admits an $H$-reduction.
By Corollary~\ref{cor:residual-trivial}, $(\aA^H,Q,\bar\alpha)$ is $Q$-equivariantly isomorphic over $\aB$ to $(\aB \otimes C(Q),Q,\id_\aB \otimes \rho)$.
For $\tau \in \rep(Q)$, evaluation at the identity element of $Q$ identifies the corresponding $\aB$-correspondence with $\aB \otimes V_\tau$.
These identifications are natural in $\tau$ and compatible with tensor products.
Consequently, we obtain the following corollary:

\begin{corollary}\label{cor:residual-functor}
Under the assumptions of Corollary~\ref{cor:residual-trivial}, there exists a unitary monoidal natural isomorphism
$\Gamma_{\aA^H} \cong \Gamma_{\mathrm{triv}}$, where
\[
    \Gamma_{\mathrm{triv}}(\tau):= \aB \otimes V_\tau,
    \qquad \tau \in \rep(Q).
\]
In particular, for each $\tau \in \rep(Q)$,
\[
    \Gamma_{\aA^H}(\tau) \cong \aB \otimes V_\tau \cong \aB^{\dim(V_\tau)}
\]
as right Hilbert $\aB$-modules.
\end{corollary}

There is also a useful reformulation directly in terms of the original C\Star-dynamical system.
For each $\tau \in \rep(Q)$, the representation $\tau \circ q:G \to U(V_\tau)$ is trivial on $H$, and
\[
    \Gamma_{\aA^H}(\tau) = \Gamma_{\aA}(\tau \circ q).
\]
Indeed, since $\tau \circ q$ is trivial on $H$, invariance under the diagonal $G$-action forces the $\aA$-component to be $H$-invariant.
Consequently, the following necessary condition for reduction is obtained.

\begin{corollary}\label{cor:module-obstruction}
If $(\aA,G,\alpha)$ admits an $H$-reduction, then, for each $\tau \in \rep(Q)$, there exists a unitary isomorphism of $\aB$-correspondences
\[
    \Gamma_\aA(\tau\circ q) \cong \aB \otimes V_\tau,
\]
where the right-hand side carries the standard $\aB$-bimodule structure.
\end{corollary}

\begin{remark}
The tensor-functorial conclusion of Corollary~\ref{cor:residual-functor}
is stronger than the objectwise condition in Corollary~\ref{cor:module-obstruction}.
Indeed, isomorphisms $\Gamma_{\aA^H}(\tau) \cong \aB\otimes V_\tau$ of $\aB$-correspondences, chosen independently for each $\tau \in \rep(Q)$, need not be compatible with the tensor structure
\[
    \Gamma_{\aA^H}(\tau) \otimes_\aB \Gamma_{\aA^H}(\upsilon)
    \to
    \Gamma_{\aA^H}(\tau \otimes \upsilon)
\]
for all $\tau,\upsilon \in \rep(Q)$.

Moreover, even the existence of a unitary monoidal natural isomorphism
$\Gamma_{\aA^H} \cong \Gamma_{\mathrm{triv}}$ does not in general imply the existence of an $H$-reduction.
As the preceding example demonstrates, the additional centrality condition in Theorem~\ref{thm:normal-reduction} is essential.
\end{remark}

\begin{corollary}\label{cor:K-obstruction}
If $(\aA,G,\alpha)$ admits an $H$-reduction, then, for each $\tau \in \rep(Q)$,
\[
    [\Gamma_{\aA}(\tau\circ q)] = \dim(V_\tau)[\aB]
\]
in $K_0(\aB)$. 
Consequently, the existence of $\tau \in \rep(Q)$ for which
\[
    [\Gamma_{\aA}(\tau \circ q)] \neq \dim(V_\tau)[\aB]
\]
obstructs reduction to $H$.
\end{corollary}

Here the $K_0$-classes are well defined because the associated correspondences of a free compact-group action are finitely generated projective as right $\aB$-modules.

\begin{example}
Let $n \geq 2$, let $G := O(n)$, and let $H := SO(n)$. 
Consider the reflection
\[
    r := \operatorname{diag}(-1,1,\ldots,1) \in G.
\]
Then $K := \{1,r\} \cong \Z_2$ is a subgroup of $G$ and $O(n) = SO(n) \rtimes K$.

Endow $M_2(\C)$ with the $K$-action $\gamma$ determined by $\gamma_r := \operatorname{Ad}(u)$,
where
\[
    u := \begin{pmatrix}
            1&0\\
            0&-1
        \end{pmatrix}.
\]
Its fixed point algebra is
\[
    \aB := M_2(\C)^K
    =\left\{
        \begin{pmatrix}
            a & 0\\
            0 & b
        \end{pmatrix}:a,b \in \C
    \right\}
    \cong
    \C^2.
\]
The homogeneous components of the $K$-action are
\[
    M_2(\C)_0 = \operatorname{span}\{e_{11},e_{22}\},
    \qquad
    M_2(\C)_1 = \operatorname{span}\{e_{12},e_{21}\}.
\]
Since
$M_2(\C)_1 M_2(\C)_1 = M_2(\C)_0$, the C\Star-dynamical system $(M_2(\C),K,\gamma)$ is free.

Now let
\[
    \aA := \operatorname{Ind}_K^{O(n)}(M_2(\C)).
\]
By Proposition~\ref{prop:induction-free}, the C\Star-dynamical system $(\aA,O(n),\alpha)$ is free and $\aA^{O(n)} = \aB$.

The $SO(n)$-fixed point algebra of $\aA$ can be identified with $M_2(\C)$:
Indeed, since each $g \in O(n)$ can be written uniquely as $g = kh$ with $k \in K$ and $h \in SO(n)$, the covariance relation gives $f(kh) = \gamma_k(f(h))$ for all $k \in K$ and $h \in SO(n)$.
Hence each $f \in \aA$ is uniquely determined by its restriction to $SO(n)$.
Conversely, for each $f \in C(SO(n),M_2(\C))$, the formula $kh \mapsto \gamma_k(f(h))$
defines a unique element in $\aA$.
Thus restriction induces a \Star-isomorphism
\[
    \aA \cong C(SO(n),M_2(\C)).
\]
Under this identification, the restricted $SO(n)$-action is given by $(\alpha_hf)(h') = f(h'h)$ for all $h,h'\in SO(n)$.

If $(\aA,O(n),\alpha)$ admitted an $SO(n)$-reduction, then Corollary~\ref{cor:residual-trivial} would imply that the residual $\Z_2$-C\Star-dynamical system is $\Z_2$-equivariantly isomorphic over $\aB$ to
\[
    (\aB \otimes C(\Z_2),\Z_2,\operatorname{id}_{\aB} \otimes \rho).
\]
However, $\aB \otimes C(\Z_2) \cong \C^4$ is commutative, whereas $\aA^{SO(n)} \cong M_2(\C)$ is not.
It follows that $(\aA,O(n),\alpha)$ admits no $SO(n)$-reduction.

Notice that the $K$-theoretic obstruction of Corollary~\ref{cor:K-obstruction} does not detect this example: the nontrivial spectral
subspace is free of rank one as a right $\aB$-module.
However, as an $\aB$-correspondence, this spectral subspace is not isomorphic to $\aB$, since its left action exchanges the two minimal central projections of $\aB \cong \C^2$.
Thus the correspondence obstruction of Corollary~\ref{cor:module-obstruction} detects the failure of reduction, whereas the $K$-theoretic obstruction does not.
\end{example}

\section*{Open problems}

The results above suggest several directions for further study.

\subsection*{Noncommutative principal \texorpdfstring{$O(n)$}{O(n)}-bundles.}

Construct and analyze further examples of noncommutative principal $O(n)$-bundles, particularly examples for which the existence or nonexistence of an $SO(n)$-reduction can be determined through the associated tensor functor.
Of particular interest is the extent to which tensor-functorial obstructions can be made explicit.

\subsection*{Bundles with homogeneous fibres.}

Extend the present approach to bundles with homogeneous fibres.
An algebraic framework for such bundles, based on principal coalgebra extensions and cotensor products, was developed in~\cite{BrSz19}.
In the present C\Star-algebraic setting, for an arbitrary closed subgroup $H \subseteq G$, the inclusion $\aA^G \subseteq \aA^H$ is the natural analogue of the classical homogeneous bundle $P/H \to X$ with fibre $G/H$.
If $H$ is normal, the algebra $\aA^H$ carries the residual $G/H$-action considered above.
For non-normal subgroups, no quotient-group action is available, and it would be interesting to develop a C\Star-dynamical and tensor-functorial description of the corresponding homogeneous bundle and to relate it to the algebraic framework of~\cite{BrSz19}.

In particular, one may ask for a noncommutative analogue of the classical section criterion
\[
    P\text{ admits an }H\text{-reduction}
    \qquad
    \Longleftrightarrow
    \qquad
    P/H\to X\text{ admits a section}.
\]
Under Gelfand duality, a section corresponds to a \Star-homomorphism $C(P/H) \to C(X)$ splitting the inclusion $C(X) \to C(P/H)$.
In the noncommutative setting, the corresponding candidate is a unital \Star-homomorphism
\[
    r: \aA^H \to \aB,
    \qquad 
    r\vert_\aB = \id_\aB.
\]
As the preceding example demonstrate, the existence of such a splitting alone need not imply reduction.
For normal subgroups, Theorem~\ref{thm:normal-reduction} identifies an
additional centrality condition through an equivariant homomorphism from $C(G/H)$ into $Z(\aA)$.
For arbitrary closed subgroups, it remains interesting to determine which additional compatibility conditions on a splitting characterize reduction.

\enlargethispage{\baselineskip}
\bibliographystyle{abbrv}
\bibliography{short,RedNCG}

\begin{thebibliography}{10}

\bibitem{BaCoHa15}
P.~Baum, K.~D. Commer, and P.~M. Hajac.
\newblock Free actions of compact quantum groups on unital {C$^*$}-algebras.
\newblock {\em Doc. Math.}, 22:825--849, 2017.

\bibitem{BrSz19}
T.~Brzezi{\'n}ski and W.~Szyma{\'n}ski.
\newblock An algebraic framework for noncommutative bundles with homogeneous
  fibres.
\newblock {\em Algebra Number Theory}, 15(1):217--240, 2021.

\bibitem{Cac23}
B.~{\'C}a{\'c}i{\'c}.
\newblock Geometric foundations for classical {{\(\mathrm{U}(1)\)}}-gauge
  theory on noncommutative manifolds.
\newblock {\em Commun. Math. Phys.}, 405(9):90, 2024.
\newblock Id/No 209.

\bibitem{CacMes19}
B.~{\'C}a{\'c}i{\'c} and B.~Mesland.
\newblock Gauge theory on noncommutative {R}iemannian principal bundles.
\newblock {\em Comm. Math. Phys}, 388(1):107--198, 2021.

\bibitem{Ell00}
D.~A. Ellwood.
\newblock A new characterisation of principal actions.
\newblock {\em J. Func. Anal.}, 173(1):49--60, 2000.

\bibitem{FaLa24}
C.~Farsi and F.~Latr\'emoli\`ere.
\newblock Collapse in noncommutative geometry and spectral continuity.
\newblock preprint, arXiv:2404.00240v1, accepted for publication in J. London
  Math., 2026.

\bibitem{LaSu05}
G.~Landi and W.~v. Suijlekom.
\newblock Principal fibrations from noncommutative spheres.
\newblock {\em Comm. Math. Phys}, 260(1):203--225, 2005.

\bibitem{Ne13}
S.~Neshveyev.
\newblock Duality theory for nonergodic actions.
\newblock {\em M{\"u}nster J. Math.}, 7(2):414--437, 2013.

\bibitem{Phi87}
C.~N. Phillips.
\newblock {\em Equivariant {K}-Theory and Freeness of Group Actions on
  {C$^*$}-Algebras}, volume 1274 of {\em Lecture Notes in Math.}
\newblock Springer, 1987.

\bibitem{Rieffel91}
M.~A. Rieffel.
\newblock Proper actions of groups on {C$^*$}-algebras.
\newblock In {\em Mappings of Operator Algebras}, volume~84 of {\em Progr.
  Math.}, pages 141--182. Birkh{\"a}user, 1991.

\bibitem{SchWa17}
K.~Schwieger and S.~Wagner.
\newblock {Part III}, {Free} actions of compact quantum groups on
  {C$^*$}-algebras.
\newblock {\em SIGMA}, 13(062):19 pages, 2017.

\bibitem{SchWa20}
K.~Schwieger and S.~Wagner.
\newblock Lifting spectral triples to noncommutative principal bundles.
\newblock {\em Adv. Math.}, 396:108160, 2022.

\bibitem{SchWa24}
K.~Schwieger and S.~Wagner.
\newblock Realizations of free actions via their fixed point algebras.
\newblock {\em Canad. Math. Bull.}, page 1–14, 2025.

\bibitem{Du95}
M.~{\DJ}ur{\dj}evi{\'c}.
\newblock Quantum principal bundles and {T}annaka-{K}rein duality theory.
\newblock In {\em Proceedings of the {XXVIII} {S}ymposium on {M}athematical
  {P}hysics ({T}oru\'n, 1995)}, volume~38, pages 313--324, 1996.

\bibitem{Wa23}
S.~Wagner.
\newblock On noncommutative frame bundles.
\newblock {\em New York J. Math.}, 31:837--856, 2025.

\bibitem{Wa25}
S.~Wagner.
\newblock Noncommutative principal bundles and central extensions.
\newblock arXiv:2509.02263v1, accepted for publication in Int. Math. Res. Not.
  IMRN, 2026.

\end{thebibliography}

\end{document}